\documentclass[10pt]{amsart}
\usepackage{amsmath, amssymb}
 \usepackage{mathrsfs}
\newcommand{\no}[1]{#1}
\renewcommand{\no}[1]{}
\no{\usepackage{times}\usepackage[subscriptcorrection, slantedGreek, nofontinfo]{mtpro}
\renewcommand{\Delta}{\upDelta}}
\usepackage{color}

\date{\today}
\newtheorem{theorem}{Theorem}[section]

\newtheorem{lemma}{Lemma}[section]

\newtheorem{corollary}{Corollary}[section]

\theoremstyle{remark}
\newtheorem{remark}{Remark}[section]

\numberwithin{equation}{section}

\title[Determining the initial heat distribution]{Various stability estimates for the problem of determining  an initial heat distribution from a single measurement}

\author[Mourad Choulli]{Mourad Choulli}
\address{Institut \'Elie Cartan de Lorraine, UMR CNRS 7502, Universit\'e de Lorraine, Boulevard des Aiguillettes, BP 70239, 54506 Vandoeuvre les Nancy cedex - Ile du Saulcy, 57045 Metz cedex 01, France}
\email{mourad.choulli@univ-lorraine.fr}

\date{}

\begin{document}

\begin{abstract}
We consider the problem of determining the initial heat distribution in the heat equation from a point measurement. We show that this inverse problem is naturally related to the one of recovering the coefficients of Dirichlet series from its sum. Taking the advantage of existing literature on Dirichlet series, in connection with  M\"untz's theorem, we establish various stability estimates of H\"older and logarithmic type. These stability estimates are then used to derive the corresponding ones for the original inverse problem, mainly in the case of one space dimension. 
\\
In higher space dimensions, we are  interested to an internal or  a boundary measurement. This issue is closely related to the problem of observability arising in Control Theory. We complete and improve the existing results.


\end{abstract}

\maketitle

\tableofcontents


\section{Introduction}\label{s1}

Inverse heat source problems appear in many branches of engineering and science. A typical application is for instance an accurate estimation of a pollutant source, which is  a crucial environmental safeguard in cities with dense populations (e.g. for instance \cite{EHH} and \cite{LTCW}). 

\smallskip
These  inverse problems are severely ill-posed, involving a strongly time-irreversible parabolic dynamics. Their mathematical analysis is difficult and still a widely open subject.

\smallskip
In the present work, we are concerned with sources located at the initial time. 


\subsection{State of art}

\smallskip 
{\it Approximation:} In an earlier paper Gilliam and Martin \cite{GM} considered the problem of recovering the initial data of the heat equation when the output is measured at points discrete in time and space. They observed that this problem is linked to the theory of Dirichlet series and a solution is found in the one dimensional case.

\smallskip
In \cite{GLM1}, Gilliam, Lund and Martin provide a simple and extremely accurate procedure for approximating the initial temperature for the heat equation on the line using a discrete time and spatial sampling. This procedure in based on ``sinc expansion''. Later, in \cite{GLM2}, the same authors give a discrete sampling scheme for the approximate recovery of initial data for one dimensional parabolic initial-boundary value problems on bounded interval.

\smallskip
The problem of recovering the initial states of distributed parameter systems, governed by linear partial differential equations, from finite approximate data, was studied by Gilliam, Mair and Martin in \cite{GMM}.

\smallskip
Li, Osher and Tsai considered in \cite{LOT} the inverse problem of finding sparse initial data from the sparsely sampled solutions of the heat equation. They prove that pointwise values of the heat solution at only a few locations are enough in an $\ell ^1$ constrained optimization to find the initial data.

\smallskip
 In a recent work, De Vore and Zuazua \cite{DZ} studied the problem of approximating accurately the initial data if the one-dimensional Dirichlet problem from finite measurements made at $(x_0,t_1),\ldots ,(x_0,t_n)$. They proved that, for suitable choices of the point location of the sensor, $x_0$,  there is a sequence $t_1<\ldots <t_n$ of time instants that guarantees the approximation with an optimal rate, of the order of $O(n^{-r})$ in the $L^2$-sense, depending on the Sobolev regularity of the datum being recovered.

\smallskip
It is worth mentioning that the orthogonality method in \cite{GLM2} leads to a best approximation in space, while in \cite{DZ} the authors show a best approximation in time.

\smallskip
{\it Uniqueness:} 
The determination of the initial distribution in the heat equation on a flat torus of arbitrary dimension was considered by Danger, Foote and Martin \cite{DFM}. They established that the observation of the solution along a geodesic  determines uniquely the initial heat distribution if and only if the geodesic is dense in the torus. Their result was obtained by using a Fourier decomposition together with results from the theory of almost periodic functions. Similar ideas were independently employed in \cite{CZ} in the context of the approximate controllability and unique continuation for the heat equation along oscillating sensor and actuator locations.

\smallskip
Let $u$ be the solution of the heat equation in the whole space $\mathbb{R}^d$. Nakamura, Saitoh and Syarif \cite{NSS} showed that the initial distribution is determined and simply represented by the observations $u(t,x_0,x')$ and $\partial _{x_1}u(t,x_0,x')$, $t\geq 0$, $x'\in \mathbb{R}^{d-1}$, for some fixed $x_0\in \mathbb{R}$.

\smallskip
{\it Stability:} To our knowledge there are only few works dealing with the stability issue. In a series of papers by Saitoh and al. \cite{SY, AS, TSS}, based on Reznitskaya transform and some properties of Bergman-Selberg spaces, they obtained Lipschitz stability estimate from a point or a boundary observation.


\subsection{The relationship with the Dirichlet series}\label{sb1.1}
We consider the initial-boundary value problem, abbreviated to IBVP in the sequel, for the one dimensional heat equation
\begin{equation}\label{1.1}
\left\{
\begin{array}{lll}
(\partial _t-\partial _x^2)u=0\;\; \mbox{in}\; (0,\pi )\times (0,+\infty ),
\\
u(0,\cdot )=u(\pi ,\cdot )=0,
\\
u(\cdot ,0)=f.
\end{array}
\right.
\end{equation}
It has a unique solution $u_f\in C([0,+\infty ),L^2((0,\pi )))$ whenever $f\in L^2((0,\pi))$. This solution can be written in term of Fourier series as follows
\begin{equation}\label{1.2}
u_f(x,t)=\frac{2}{\pi}\sum_{k\geq 1}\widehat{f}_ke^{-k^2t}\sin (kx),
\end{equation}
where $\widehat{f}_k$ is the $k$-th Fourier coefficient of $f$:
\[
\widehat{f}_k=\frac{2}{\pi}\int_0^\pi f(x)\sin (kx)dx.
\]

Given a point $x_0\in (0,\pi )$ for the placement of the sensor, we address the question of reconstructing the initial distribution $f$ from $u_f(x_0,t)$, $t\in (0,T)$. In light of \eqref{1.2}, setting $a_k=\widehat{f}_k\sin (kx_0)$, we see that the actual problem is reduced to one of recovering the sequence $a=(a_k)$ from the sum of the corresponding Dirichlet series 
\[
\sum_{k\geq 1}a_ke^{-k^2t}.
\]
For this to be the case $x_0$ has to be chosen in a strategic way so that $\sin (kx_0) \ne 0$ for all $k \ge 1$.

\subsection{Outline}
Section 2 is devoted to establish stability estimates for the problem of recovering the coefficients $a_k$ of a general Dirichlet series $\sum_{k\geq 1}a_ke^{-\lambda_k t}$ from its sum. Here $(\lambda _k)$ is a strictly increasing sequence of non negative real numbers, converging to infinity. The behavior of this problem depends on whether  the series $\sum 1/\lambda _k$ converges. 
\begin{itemize}
\item When $\sum 1/\lambda _k$ converges, $(e^{-\lambda _kt})$ admits a biorthogonal family and this simplifies the analysis.

\item The case $\sum \frac{1}{\lambda _k}=\infty$ is much harder and the stability estimates we obtain are weaker. 
\end{itemize}
In both cases, we prove stability estimate of logarithmic or H\"older type and to carry out our analysis we need to compare $\lambda _k$ with $k^\beta$, $\beta >0$ is given. In the case $\sum \frac{1}{\lambda _k}=\infty$, a gap condition  on the sequence $(\lambda _k)$ is also necessary in our analysis (see \eqref{2.12} in subsection \ref{sb2.2}).

\smallskip
In Section 3, we apply the results obtained in the Section 2 to the problem of determining the initial heat distribution in a one dimensional heat equation from an overspecified data. We get a various stability estimates of H\"older or logarithmic type. Our results include fractional heat equations of any order. We also establish a boundary observability inequality with an output located at one of the end points. This later enables us to obtain a logarithmic stability estimate for the inverse problem of recovering the initial condition from a boundary measurement.  

\smallskip
In  Section 4, we first consider the particular case of the (fractional) heat equation in a d-dimensional rectangle $\Omega$.  We show that if the Dirichlet eigenvalues of the laplacian in $\Omega$ are simple\footnote{This is a generic property among open bounded subsets of $\mathbb{R}^d$.}, then,  in some cases, the determination of the initial heat distribution can be reduced to the one dimensional case. This is achieved when the measurements consist in the values of the solution of the heat equation on $d$ affine $(d-1)$-dimensional subspaces. Next, we revisit the case when overdetermined data is an internal or a boundary observation. We comment the existing results and show that these laters can be improved using  recent observability inequalities.

\subsection{Notations} The unit ball of a Banach space $X$ is denoted by $B_X$. 

\smallskip
$\ell ^p=\ell^p (\mathbb{C})$, $1\leq p<\infty$, is the usual Banach space of complex-valued sequences $a=(a_n)$ such that the series $\sum |a_n|^p$ is convergent. We equip $\ell ^p$ with its natural norm
\[
\|a\|_{\ell ^p}=\left(\sum_{k\geq 1}|a_k|^p\right)^{1/p},\;\; a=(a_k)\in \ell ^p.
\]

$\ell ^\infty=\ell^\infty (\mathbb{C})$ denotes the usual Banach space of bounded complex-valued sequences $a=(a_n)$, normed by
\[
\|a\|_\infty =\sup_k|a_k|,\;\; a=(a_k)\in \ell ^\infty .
\]

For $\theta >0$,  the space $h^\theta =h^\theta (\mathbb{C})$ is defined as follows
\[
h^\theta  =\left\{b=(b_k)\in \ell ^2;\; \sum_{k\geq 1} \langle k\rangle ^{2\theta }|b_k|^2<\infty\right\},
\]
where $\langle k\rangle =(1+k^2)^{1/2}$.

\smallskip
$h^\theta $ is a Hilbert space when it is equipped with the norm
\[
\|b\|_{h^\theta }=\left( \sum_{k\geq 1} \langle k\rangle ^{2\theta }|b_k|^2 \right)^{1/2}.
\]


\section{Determining the coefficients of a Dirichlet series from its sum}\label{s2}

We limit our study to Dirichlet series whose coefficients consist in sequences from $\ell ^p$, $1\leq p\leq \infty$. 

\smallskip
We pick a real-valued  sequence $(\lambda _k)$ satisfying $0<\lambda _1<\lambda _2< \cdots \lambda _k< \ldots$ and $\lambda _k\rightarrow +\infty$ as $k$ goes to $+\infty$.
 
\smallskip
To $a=(a_n)\in \ell ^p$, $1\leq p\leq \infty$, we associate the Dirichlet series
\[
F_a(t)=\sum_{n\geq 1}a_ne^{-\lambda _nt}.
\]
Let $1<p<\infty$ and $p'$ be the conjugate exponent of $p$. Observing that 
\[
a_ne^{-\lambda _nt}=a_ne^{-\lambda _nt/p}e^{-\lambda _nt/p'},
\]
we get by applying H\"older's inequality
\begin{align*}
|F_a(t)|&\leq \left(\sum_{n\geq 1}e^{-\lambda _nt}\right)^{1/p'}\sum_{n\geq 1}|a_n|^pe^{-\lambda _nt}
\\
&\leq \left(\sum_{n\geq 1}e^{-\lambda _nt}\right)^{1/p'}\sum_{n\geq 1}|a_n|^p,\;\; t>0.
\end{align*}
Also,
\begin{align*}
&|F_a(t)|\leq \sum_{n\geq 1}|a_n|,\;\; t\geq 0\; (p=1),
\\
&|F_a(t)|\leq \left(\sum_{n\geq 1}e^{-\lambda _nt}\right)\sup_n|a_n|,\;\; t> 0\; (p=\infty ).
\end{align*}
Therefore, the series $F_a(t)$ converges for $t>0$, for any $a\in \ell ^p$, $1\leq p\leq \infty$. 

\smallskip
From the classical theory of Dirichlet series (see for instance \cite{Wi}), we know that $F_a$ has an analytic extension to the half plane $\Re z>0$. Since a Dirichlet series is zero if and only if its coefficients are identically equal to zero, we conclude that the knowledge of $F_a$ in a subset of $\Re z>0$ having an accumulation point determines uniquely $a$. In other words, if $D\subset \{\Re z>0\}$ has an accumulation point and $F_a$ vanishes on $D$, then $a=0$.

\smallskip
The most interesting case is when $p=1$. We can define in that case the operator $\mathcal{U}$ by
\begin{align*}
\mathcal{U}:\ell ^1&\longrightarrow C_b([0,+\infty ))
\\ a&\longmapsto \mathcal{U}(a):=F_a,\; F_a(s)=\sum_{k\geq 1}a_ke^{-\lambda_ks}.
\end{align*}
Here, $C_b([0,+\infty ))$ is the Banach space of bounded continuous function on $[0,+\infty )$, equipped with the supremum norm
\[
\|F\|_\infty =\sup\{|F(s)|;\; s\in [0,+\infty )\}.
\]
Then $\mathcal{U}$ is an injective linear contractive operator.

\smallskip
Similarly to entire series, we address the question to know whether it is possible to reconstruct the coefficients of a Dirichlet series from its sum. This is always possible if the values of the sum is known in the half plane $\Re z>0$. More specifically, we have the following formula (see a proof in \cite{Wi}): for $\lambda _n<\lambda <\lambda _{n+1}$ and $\gamma >0$,
\[
\sum_{k=1}^na_k=\frac{1}{2i\pi}\mbox{pv}\int_{\gamma -i\infty}^{\gamma +i\infty}\frac{F_a(z)}{z}e^{\lambda z}dz.
\]

\smallskip
In the present section we aim to establish the modulus of continuity, at the origin, of the inverse of the mapping $a\in \ell ^p \rightarrow F_a{_{|D}}$, where $D$ is a subset of  $(0,+\infty)$. Rouhly speaking, we seek an estimate of the form
\[
\| a\|_{\ell ^p}\leq \Psi \left(\|F_a\|_{L^\infty (D)}\right),
\]
for $a$ in some appropriate subset of $\ell ^p$, $p=1$ or $2$,  where $\Psi$ is a continuous, non decreasing and non negative real-valued function satisfying $\Psi (0)=0$.
 
 \smallskip
 We discuss  separately the cases $\Lambda :=\sum_{k\geq 1}\frac{1}{\lambda _k}<\infty$ and $\Lambda :=\sum_{k\geq 1}\frac{1}{\lambda _k}=\infty$.
 
 \smallskip
It is worthwhile recalling  that the nature of the series $\sum_{k\geq 1}\frac{1}{\lambda _k}$ is related to M\"untz's theorem saying that the closure of the vector space spanned by $\{e^{-\lambda _kt};\; k\geq 1\}$ is dense in $C_0([0,+\infty ))=\{\varphi \in C_0([0,+\infty ));\; \varphi (+\infty )=0\}$  if and only if $\Lambda =\infty$. Usually this theorem is stated in the following equivalent form: $\{x^{\lambda _k};\;\; k\geq 1\}$ is dense in $C_0([0,1])=\{\varphi \in C_0([0,1]);\; \varphi (0)=0\}$ if and only if $\Lambda=\infty$.

\smallskip
Let $(\lambda _k)$ be the sequence of the Dirichlet eingenvalues of the Laplacian on a bounded domain of $\mathbb{R}^d$, that we assumed simple. From \cite[Lemma 3.1, page 229]{Ka}, $\lambda _k=O(k^{2/d})$. Therefore, in that case $\Lambda <\infty$ holds if and only if $d=1$.

 
 \subsection{The case $\Lambda <\infty$}\label{sb2.1}
 
 The following lemma will be useful in the sequel. Henceforth, 
 \[
 F_a^N(t)=\sum_{k=1}^N a_ke^{-\lambda _kt},\;\; t\geq 0,\;\; N\geq 1.
 \]
 
 \begin{lemma}\label{lemma2.1}
 Let $a\in \ell ^2$. Then $F_a\in L^2((0,T))$ and $F_a^N$ converges to $F_a$ in $L^2((0,T))$ as $N\rightarrow \infty$.
 \end{lemma}
 
 \begin{proof}
From Cauchy-Schwarz's inequality 
 \[
\left|F_a^N(t)\right|^2,\; |F_a(t)|^2\le \left(\sum_{k=1}^\infty \left|a_k\right|^2\right)\left(\sum_{k=1}^\infty e^{-2\lambda _kt}dt\right)=G(t),\;\; N\ge 1.
 \]
 But
 \[
 \sum_{k=1}^\infty \int_0^Te^{-2\lambda _kt}dt= \sum_{k=1}^\infty \frac{1-e^{-2\lambda _kT}}{\lambda _k}\le \frac{\Lambda}{2}.
 \]
 Hence $F_a\in L^2((0,T))$ and 
 \[
 \|F_a\|_{L^2((0,T))}\le  \frac{\Lambda}{2}\|a\|_{\ell ^2}.
 \]

Similarly
\[
\left| F_a(t)-F_a^N(t)\right|^2\le \left(\sum_{k\geq 1}e^{-2\lambda _kt}\right)\sum_{k\geq N+1}|a_k|^2,\;\; N\geq 1,\; t>0.
\]
Thus $F_a^N(t)\rightarrow F_a(t)$ as $N\rightarrow \infty$, for any $t\in (0,T]$.  As $\left|F_a^N\right|^2 \le G$, $N\ge 1$, we apply Lebesgue's dominated convergence theorem in order to get that $F_a^N$ converges to $F_a$ in $L^2((0,T))$. 
 \end{proof}
 
 Let $\mathscr{E}$ be the closure in $L^2((0,T))$ of the vector space spanned by $\{e^{-\lambda _kt};\; k\geq 1\}$. By \cite[theorem in page 24]{Sc}, $\mathscr{E}$ is a proper subspace of $L^2((0,T))$. Additionally, $\{e^{-\lambda _kt}\}$ possesses a biorthogonal set $\{\psi _k\}$ in $L^2((0,T))$: 
 \begin{equation}
 \int_0^T\psi _k(t)e^{-\lambda _nt}dt=\delta _{nk}.
 \end{equation}\label{2.3}
We assume that there are constants $\beta >1$, $K>0$ and $\alpha >0$ such that
 \begin{equation}\label{2.1}
 \lambda _n=K(n+\alpha  )^\beta +o(n^{\beta -1}),\;\; n\rightarrow \infty .
 \end{equation}
From in \cite[formula $(3.25)$]{FR}, 
\begin{equation}\label{2.2}
\|\psi _n\|_{L^2((0,T))}\leq Ce^{C\lambda _n^{1/\beta}},\;\; n\ge 1,
\end{equation}
for some constant $C>0$ that can depend only on $(\lambda _n)$.

\smallskip
In light of \eqref{2.3}, we get
\[
a_n=\int_0^T F_a^N(t)\psi_n(t)dt\;\; n\leq N.
\]
Therefore, by Cauchy-Schwarz's inequality, 
\begin{equation}\label{2.3.1}
|a_n| \leq \|F_a^N\|_{L^2((0,T))}\|\psi _n\|_{L^2((0,T))},\;\; n\leq N.
\end{equation}
By Lemma \ref{lemma2.1}, we can pass to the limit, as $N\rightarrow \infty$, in \eqref{2.3.1}. We get 
\[
|a_n|\leq \|F_a\|_{L^2((0,T))}\|\psi _n\|_{L^2((0,T))},\;\;n\ge 1.
\]
Then, \eqref{2.2} implies
\begin{equation}\label{2.2.1}
|a_n|\leq Ce^{C\lambda _n^{1/\beta}}\|F_a\|_{L^2((0,T))}.
\end{equation}
Hence, for any $N\geq 1$, 
\[
\sum_{n=1}^N|a_n|^2\leq CNe^{C\lambda _N^{1/\beta}}\|F_a\|^2_{L^2((0,T))}.
\]
But $\lambda _N^{1/\beta}=O(N)$ by \eqref{2.1}. Consequently,
\begin{equation}\label{2.4}
\sum_{n=1}^N|a_n|^2\leq e^{CN}\|F_a\|^2_{L^2((0,T))}.
\end{equation}

Let $\theta >0$ and $m>0$ be two given constants. By inequality \eqref{2.4} we have, for any $a\in mB_{h^\theta }$,
 \begin{align*}
 \|a\|_{\ell ^2}^2&=\sum_{k=1}^N|a_k|^2+\sum_{k>N}|a_k|^2
 \\
 &\leq \sum_{k=1}^N|a_k|^2+\frac{1}{\langle N+1\rangle ^{2\theta }}\sum_{k>N} \langle k\rangle ^{2\theta }|a_k|^2
 \\
 &\leq e^{CN}\|F_a\|^2_{L^2((0,T))}+\frac{m^2}{N^{2\theta }}.
 \end{align*}
 That is,
 \begin{equation}\label{2.5}
  \|a\|_{\ell ^2}^2 \leq e^{CN}\|F_a\|^2_{L^2((0,T))}+\frac{m^2}{N^{2\theta }}.
 \end{equation}

We need the following result to pursue our analysis. It is stated as Theorem 5.1 in \cite{BE}.
\begin{theorem}\label{theorem2.1}
Let $0<\tau <1$. There is a constant $c$ depending only on $(\lambda_k)$ (and not on $\rho$, $A$ and the length of $p$) so that 
\[
\|p\|_{L^\infty (0,\rho )}\leq c\|p\|_{L^\infty (A)},
\]
for every $p\in \mbox{span}\{x^{\lambda _k};\; k\geq 1\}$ and every Lebesgue-measurable set $A\subset [\rho ,1]$ of Lebesgue measure at least $\tau$.
\end{theorem}

\begin{corollary}\label{corollary2.1}
Let $B$ be a Lebesgue-measurable set of $[0,T]$ of positive Lebesgue measure. There is a constant $d$, that can depend on $(\lambda _k)$, $B$ and $T$, so that,  for any $a\in \ell^1$,  
\begin{equation}\label{2.6}
\|F_a\|_{L^\infty ((0,T))}\leq d\|F_a\|_{L^\infty (B)}.
\end{equation}
\end{corollary}

\begin{proof}
We proceed similarly as in the beginning of the proof of  \cite[Corollary 5.2]{MRT}. Let $\rho =e^{-T}$ and 
\[
A=\{ x=e^{-t};\; t\in B\}\subset [\rho ,1].
\]
Then
\[
|A|=\int_Be^{-t}dt\geq e^{-T}|B|.
\]
We recall that
\[
F_a^N(t)=\sum_{k=1}^Na_ke^{-\lambda _kt},\;\; t\geq 0,\; N\geq 1.
\]
We get by applying Theorem \ref{theorem2.1} 
\begin{equation}\label{2.6.1}
\|F_a^N\|_{L^\infty ((0,T))}\leq d\|F_a^N\|_{L^\infty (B)},\;\; \textrm{for any}\; N\geq 1.
\end{equation}
On the other hand, $F_a^N$ converges uniformly to $F_a$ in $[0,+\infty )$. This is an immediate consequence of the following estimate
\[
\left| F_a(t)-F_a^N(t)\right|\leq \sum_{k\geq N+1}\left| a_k\right|,\;\; N\geq 1,\; t\geq 0.
\]
Therefore, \eqref{2.6} is obtained by passing to the limit, as $N\rightarrow \infty$, in \eqref{2.6.1}.
\end{proof}

Now, estimate \eqref{2.6} in \eqref{2.5} yields, where $B$ is a given Lebesgue-measurable set of $[0,T]$ of positive Lebesgue measure,
 \begin{align}
 \|a\|_{\ell ^2}^2 &\le e^{CN}\|F_a\|^2_{L^\infty (B)}+\frac{m^2}{N^{2\theta }},\;\; a\in mB_{h^\theta}\cap \ell ^1 \label{2.7}
 \\
 &\le \max (1,m^2)\left(  e^{CN}\|F_a\|^2_{L^\infty (B)}+\frac{1}{N^{2\theta }}\right).\nonumber
 \end{align}
 
 Let $\widetilde{N}$ be the greatest integer satisfying
 \[
 e^{C\widetilde{N}}\|F_a\|^2_{L^\infty (B)}\leq \frac{1}{\widetilde{N}^{2\theta }}.
 \]
 Such an $\widetilde{N}$ exists provided that $\|F_a\|_{L^\infty (B)}$ is sufficiently small. A straightforward computation shows that 
 \[
 \widetilde{N}>C\left|\ln \|F_a\|_{L^\infty (B)} \right|,
 \]
By taking $N=\widetilde{N}$ in \eqref{2.7}, we get that there exist $\delta >0$ so that
 \begin{equation}\label{2.8}
 \|a\|_{\ell ^2} \leq C\left|\ln \|F_a\|_{L^\infty (B)} \right|^{-\theta },\;\; \textrm{if}\;\;  \|F_a\|_{L^\infty (B)}\leq \delta.
 \end{equation}
 
 When $\|F_a\|_{L^\infty (B)}>\delta$, 
 \begin{equation}\label{2.8.1}
 \|a\|_{\ell ^2}\leq \frac{m}{\delta} \|F_a\|_{L^\infty (B)}.
 \end{equation}
 
 A combination of \eqref{2.8} and \eqref{2.8.1} implies
\[
 \|a\|_{\ell ^2} \leq C\left\{ \left|\ln \|F_a\|_{L^\infty (B)} \right|^{-\theta }+\|F_a\|_{L^\infty (B)}\right\}.
\]

\smallskip
We sum up our analysis in the following theorem.

\begin{theorem}\label{theorem2.2}
We assume that assumption \eqref{2.1} is satisfied. Let $B$ a Lebesgue-measurable set of $[0,T]$ of positive Lebesgue measure, $m>0$ and $\theta  >0$. There exists a constant $C>0$, that can depend on $B$, $(\lambda _n)$, $m$ and $\theta$, so that, for any $a\in mB_{h^\theta }\cap \ell ^1$,
\begin{equation}\label{2.14}
\|a\|_{\ell ^2} \leq C\left\{ \left|\ln \|F_a\|_{L^\infty (B)} \right|^{-\theta }+\|F_a\|_{L^\infty (B)}\right\}.
 \end{equation}
\end{theorem}

We observe that $h^\theta \subset \ell ^1$ when $\theta >1/2$. Therefore, $mB_{h^\theta}\cap \ell ^1=mB_{h^\theta}$ if $\theta >1/2$.

\begin{remark}\label{remark2.1}
1) In light of \eqref{2.1} and \eqref{2.2.1}, we have the following Lipschitz stability estimate, where $c$ is a constant depending only on $(\lambda _n)$, \[ \sum_{n\geq 1}e^{-cn}|a_n|^2\leq C \|F_a\|_{L^2((0,T))},\;\; a\in \ell ^2.\] Here the left hand side of this inequality is seen as a an $\ell^2$-weighted norm of $a$.

\smallskip
2) Starting from \eqref{2.5}, we can prove the following estimate 
\begin{equation}\label{2.14.1}
\|a\|_{\ell ^2} \leq C\left\{ \left|\ln \|F_a\|_{L^2((0,T))} \right|^{-\theta }+\|F_a\|_{L^2((0,T))}\right\},\;\; a\in mB_{h^\theta},
\end{equation}
for any $\theta >0$.

\smallskip
3) It is possible to establish a H\"older stability estimate. This is can be done by substituting in Theorem \ref{theorem2.2} $h^\theta$ by the following subspace
\[
h_{c, \gamma}=\left\{b=(b_n);\; \sum_{n\geq 1}e^{cn^\gamma}|b_n|^2<\infty\right\},
\]
with $c>0$ and $\gamma >1$. A  proof of a similar result will be detailed in the next subsection.

\smallskip
4) A Lipschitz or a H\"older stability estimate is not true in general. Indeed, let us assume that we have an estimate of the form, where $0<\mu \leq 1$,
\begin{equation}\label{2.15.1}
\| a\|_{\ell ^2}\leq C\left( \|F_a\|_{L^2((0,T))}^\mu +\|F_a\|_{L^2((0,T))}\right),\;\; a\in B_{h^\theta} .
\end{equation}
Let $(e_k)$ be the usual orthonormal basis of $\ell ^2$. That is $e_k=(\delta _{kn})$, where $\delta_{kn}$ is the Kronecker symbol. Letting $f_k=\langle k\rangle ^{-\theta} e_k$, we get by a straightforward computation 
\begin{equation}\label{2.15.2}
\| F_{f_k}\|_{L^2((0,T))}\le \frac{1}{\langle k\rangle ^\theta}\frac{1}{\sqrt{2\lambda _k}},\;\; k\geq 1.
\end{equation}
Since $f_k\in B_{h^\theta}$, if \eqref{2.15.1} is true then we would have from \eqref{2.15.2}
\begin{equation}\label{2.15.3}
\frac{1}{\langle k\rangle ^\theta}\leq C\left( \frac{1}{\langle k\rangle ^{\theta \mu}\lambda _k^{\mu /2}}+\frac{1}{\langle k\rangle ^\theta\lambda _k^{1/2}}\right),\;\;  k\geq k_0,
\end{equation}
The particular choice of $\lambda _k=k^\beta$, $k\geq 1$ in \eqref{2.15.1} yields
\[
1\leq C\left( \frac{1}{k^{\mu \beta /2 -\theta (1-\mu)}}+\frac{1}{k^{\beta/2}}\right),\;\;  k\geq 1,
\]
But this inequality cannot be true if $\mu \beta /2 -\theta (1-\mu)>0$.
\end{remark}


\subsection{The case $\Lambda =\infty$}\label{sb2.2}

We pick $a=(a_k)\in B_{\ell^1}$ and we set 
\[
\varrho =\|F_a\|_\infty \; (\leq \|a\|_{\ell ^1}\leq 1).
\]
Since $\varrho \geq |F_a(s)|\geq |a_1|e^{-\lambda _1s}-e^{-\lambda _2s}$, we have
\[
|a_1|\leq \varrho e^{\lambda _1s}+e^{-(\lambda _2-\lambda _1)s},\; \mbox{for any}\; s\geq 0.
\]
The choice of $s=\frac{1}{\lambda _2}\ln(1/\varrho )$ gives
\[
|a_1|\leq 2\varrho^{1-\frac{\lambda _1}{\lambda _2}}.
\]
More generally, we have
\[
|a_k|\leq (\varrho +|a_1|+\ldots +|a_{k-1}|)e^{-\lambda _ks}+ e^{-\lambda_{k+1}s}
\]
and then
\[
|a_k|\leq (\varrho +|a_1|+\ldots +|a_{k-1}|)^{1-\frac{\lambda _k}{\lambda _{k+1}}}.
\]
So an induction argument leads to the following estimate
\begin{equation}\label{2.9}
|a_1|+\ldots +|a_k|\leq C_k\varrho^{(1-\frac{\lambda _1}{\lambda _2})\ldots (1-\frac{\lambda _k}{\lambda _{k+1}})},
\end{equation}
with $C_1=2$ and $C_{k+1}=3C_k+2$. Therefore $C_k=2\sum_{i=1}^{k-1} 3^i\leq 3^k$, $k\geq 2$.

\smallskip
If 
\[
p_k=\prod_{i=1}^k\left( 1-\frac{\lambda _i}{\lambda _{i+1}}\right),
\]
then \eqref{2.9} implies
\begin{equation}\label{2.10}
\sum_{i=1}^k|a_i| \leq 3^k\varrho ^{p_k}.
\end{equation}

We introduce the weighted $\ell ^1$ space, where  $\theta  >0$, 
\[
\ell ^{1,\theta } =\left\{ a=(a_i);\; \sum_{i\geq 1}i^\theta  |a_i|<\infty \right\}.
\]
We equip $\ell ^{1,\theta }$ with its natural norm
\[
\| a\|_{\ell ^{1,\theta }}=\sum_{i\geq 1}i^\theta  |a_i|.
\]

Let $a\in B_{\ell ^{1,\theta }}$. In light of \eqref{2.10}, we have
\begin{align*}
\|a\|_{\ell ^1}&=\sum_{i=1}^k|a_i|+\sum_{i\geq k+1}|a_i|
\\
&\leq 3^k\varrho ^{p_k}+\frac{1}{k^\theta }\sum_{i\geq k+1}i^\theta  |a_i|.
\end{align*}
Hence
\begin{equation}\label{2.11}
\|a\|_{\ell ^1}\leq 3^k\varrho ^{p_k}+\frac{1}{k^\theta },\;\; k\geq 1.
\end{equation}

Let us assume that the sequence $(\lambda _k)$ obeys to the following assumptions: there are four constants $\beta _0\geq 0$, $\beta _1 >0$, $c>0$ and $d>0$ so that
\begin{equation}\label{2.12}
\lambda _{i+1}-\lambda _i\geq \frac{d}{(i+1)^{\beta_0}}\;\; \mbox{and}\;\;  \lambda _i\leq ci^{\beta _1} \;\; i\geq 1.
\end{equation}
Under these assumptions, 
\[
p_k\geq q_k=\frac{c_\ast^k}{(k+1)^{\beta k}},\;\; \mbox{with}\; \beta= \beta_0+\beta _1 \; \mbox{and}\; c_\ast =\min (d/c,1).
\]
Therefore, \eqref{2.11} yields
\begin{equation}\label{3.13}
\|a\|_{\ell ^1}\leq 3^k\varrho ^{q_k}+\frac{1}{k^\theta },\;\; k\geq 1.
\end{equation}
If $\varrho$ is sufficiently small, we denote by $\widetilde{k}$ the greatest positive integer such that 
\[ 
3^{\widetilde{k}}\varrho ^{q_{\widetilde{k}}}\leq  \frac{1}{\widetilde{k}^\theta }.
\]

Let $c_\star =(\ln 3+\theta+\beta )^{1/2}$. Since
\[ 
3^{\widetilde{k}+1}\varrho ^{q_{\widetilde{k}+1}} >\frac{1}{(\widetilde{k}+1)^\theta },
\]
we have
\[
\widetilde{k}>\frac{1}{2c_\star}\left( \ln |\ln \varrho | \right)^{1/2}
\]
by a straightforward computation.

\smallskip
We end up getting
\begin{equation}\label{2.12.1}
\|a\|_{\ell ^1}\le \left(\frac{1}{2c_\star}\right)^\theta  \left( \ln |\ln \|F_a\|_\infty | \right)^{-\theta  /2},\;\; \|F_a\|_\infty=\varrho \leq \varrho _0,
\end{equation}
for some $\varrho _0>0$.

\smallskip
When $\|F_a\|_\infty \ge \varrho _0$, we have
\begin{equation}\label{2.12.2}
\|a\|_{\ell ^1}\leq \|a\|_{\ell ^{1,\theta}}\leq 1\leq \frac{\|F_a\|_\infty}{\varrho _0}
\end{equation}

\smallskip
Hereafter we used that $F_{\lambda a}=\lambda F_a$, $\lambda \in \mathbb{C}$, which a consequence of the linearity of $\mathcal{U}$.  In light of \eqref{2.12.1} and \eqref{2.12.2}, we can state the following result.

\begin{theorem}\label{theorem2.4}
We assume that assumption \eqref{2.12} fulfills. Let $m>0$. There exists a constant $C>0$, that can depend only on $(\lambda_n)$ and $m$, so that, for any $a\in mB_{\ell^{1,\theta }}$,
\[
\|a\|_{\ell ^1}\leq C\left(\left| \ln  |\ln (m^{-1}\|F_a\|_\infty )|\right|^{-\theta  /2}+\|F_a\|_\infty\right).
\]
\end{theorem}

We are now going to show that, even in the present case ($\Lambda =\infty$), it is possible to establish a H\"older stability estimate.

\smallskip
We pick $a\in \ell ^1$, $N$ an non negative integer and we recall that $F_a^N$ is given by
\[
F_a^N(s)=\sum_{n=1}^{N}e^{-\lambda _ns}a_n,\;\; s\geq 0.
\]
Let $x_n=e^{-\lambda _n}$, $n=1,\ldots N$. We introduce the following Vandermonde matrix
\[
V_N=
\begin{pmatrix}
1&\ldots&1\\ x_1&\ldots&x_N\\ \vdots&\ldots&\vdots \\x_1^{N-1}&\ldots&x_N^{N-1}
\end{pmatrix}.
\]
By setting $A_N=(a_1,\ldots ,a_N)^t$ and $B_N=(F_a^N(0),\ldots ,F_a^N(N-1))^t$, we get in a straightforward manner that $V_NA_N=B_N$ or equivalently $A_N=V_N^{-1}B_N$.

\smallskip
If $V_N^{-1}= (w_{ij})$ and $\|V_N^{-1}\|=\sum_{1\leq i,j\leq N}|w_{ij}|$, then
\begin{equation}\label{2.19}
\|A_N\|_1\leq \|V_N^{-1}\|\|B_N\|_\infty .
\end{equation}

From the proof of \cite[Theorem 1]{Ga}, we obtain
\[
\|V_N^{-1}\|\leq \sum_{1\leq j\leq N}\prod_{i\neq j}\frac{1+|x_j|}{|x_i-x_j|}.
\]
Therefore, under assumption \eqref{2.12}, we get after some technical calculations 
\[
\|V_N^{-1}\|\leq Ce^{CN^{\beta _1}},
\]
where $\beta _1$ is the same as in \eqref{2.12}. Hence, \eqref{2.19} entails
\begin{equation}\label{2.20}
\sum_{n=1}^N|a_i|\leq Ce^{CN^{\beta _1}}\|F_N\|_\infty \leq Ce^{CN^{\beta _1}}\left(\|F_a\|_\infty+\sum_{n>N}|a_n|\right).
\end{equation}

For $\alpha >0$ and $\beta >0$, we introduce the following weighted $\ell ^1$-space:
\[
\ell ^1_{\alpha ,\beta }=\left\{ a=(a_i);\; \sum_{n\geq 1}e^{\alpha n^\beta}|a_i|<\infty \right\}.
\]
We equip this space with its natural norm
\[
\|u\|_{\ell ^1_{\alpha ,\beta}}=\sum_{n\geq 1}e^{\alpha n^\beta}|a_i|.
\]

Let $m$ be a non negative constant and $\beta >\beta _1$ ($\beta _1$ is the same as in \eqref{2.12}). Assuming that $a\in mB_{\ell ^1_{\alpha ,\beta}}$, we obtain in light of \eqref{2.20}
\[
\|a\|_{\ell ^1}\leq Ce^{CN^{\beta _1}}\left(\|F_a\|_\infty+me^{-cN^\beta}\right)+me^{-cN^\beta}
\]
Therefore, we find an integer $N_0$ so that for any $N\geq N_0$,
\[
\|a\|_{\ell ^1}\leq C\left(e^{CN^{\beta _1}}\|F_a\|_\infty+e^{-\widetilde{C}N^\beta}\right).
\]
We derive by minimizing with respect to $N$ the following H\"older stability estimate.
\begin{theorem}\label{theorem2.5}
We assume that  \eqref{2.12} fulfills. Let $m>0$, $\alpha >0$ and $\beta >\beta _1$. There exist two constants $C>0$ and $\gamma >0$, that can depend only on $(\lambda _n)$, $m$, $\alpha $ and $\beta$, so that, for any $a\in mB_{\ell^1_{\alpha ,\beta}}$,
\[
\|a\|_{\ell ^1}\leq C\left(\|F_a\|_\infty ^\gamma +\|F_a\|_\infty\right).
\]
\end{theorem}


\section{Determining the initial heat distribution in one the dimensional heat equation}\label{s3}

\subsection{Point measurement}\label{sb3.1}

We come back to the one dimensional heat equation. We consider again the IBVP 
\begin{equation}\label{3.1}
\left\{
\begin{array}{lll}
(\partial _t-\partial _x^2)u=0\;\; \mbox{in}\; (0,\pi )\times (0,+\infty ),
\\
u(0,\cdot )=u(\pi ,\cdot )=0,
\\
u(\cdot ,0)=f.
\end{array}
\right.
\end{equation}
The solution of the IBVP \eqref{3.1} is given by
\begin{equation}\label{3.2}
u_f(x,t)=\frac{2}{\pi}\sum_{k\geq 1}\widehat{f}_ke^{-k^2t}\sin (kx),
\end{equation}
where $\widehat{f}_k$ is the Fourier coefficient of $f\in L^2((0,\pi))$:
\[
\widehat{f}_k=\frac{2}{\pi}\int_0^\pi f(x)\sin (kx)dx.
\]

From \cite[Lemma 1.1.4, page 30]{AN}, there exists $x_0\in (0,\pi )$ satisfying
\begin{equation}\label{3.3}
|\sin (kx_0)|\geq d_0k^{-1},\;\; k\geq 1,
\end{equation}
where $d_0$ a constant depending on $x_0$.

\smallskip
Let $\theta \ge 0$. It is known that the Sobolev space $H^\theta ((0,\pi ))$ can be constructed by using Fourier series. Precisely, we have
\[
H^\theta ((0,\pi ) )=\{h\in L^2((0,\pi ));\; \sum_{k\geq 1}\langle k\rangle ^{2\theta }|\widehat{h}_k|^2<\infty \}.
\]
$H^\theta ((0,\pi ))$ is equipped with its natural norm
\[
\|h\|_{H^\theta ((0,\pi ))}=\left( \sum_{k\geq 1}\langle k\rangle ^{2\theta }|\widehat{h}_k|^2\right)^{1/2}.
\]

We take $f\in H^2((0,\pi ))$  and we set
\[
a_k=\frac{2}{\pi}\sin kx_0\widehat{f}_k,\;\; k\geq 1.
\]
In light of \eqref{3.3}, we get
\begin{equation}\label{3.4}
\frac{\pi}{2} |a_k|\leq |\widehat{f}_k|\leq  c_0k|a_k|.
\end{equation}
Here $c_0$ is a constant depending on $x_0$. Therefore
\[
\sum_{k\geq 1} |\widehat{f}_k|^2\leq c_0^2\sum_{k\geq 1}\langle k\rangle ^2 |a_k|^2 .
\]
Hence
\begin{equation}\label{3.5}
\sum_{k\geq 1} |\widehat{f}_k|^2\leq c_0^2\left( \sum_{k\geq 1}\langle k\rangle ^4|a_k|^2\right)^{1/2}\left( \sum_{k\geq 1}|a_k|^2\right)^{1/2}
\end{equation}
by Cauchy-Schwarz's inequality.

\smallskip
But
\[
\sum_{k\geq 1}\langle k\rangle ^4|a_k|^2\leq \frac{4}{\pi ^2}\sum_{k\geq 1}\langle k\rangle ^4|\widehat{f}_k|^2=\frac{4}{\pi ^2}\|f\|_{H^2((0,\pi))}^2.
\]
This estimate in \eqref{3.5} gives
\[
\|f\|_{L^2((0,\pi ))}\leq \widetilde{c}_0\|f\|_{H^2((0,\pi))}^{1/2}\|a\|_{\ell ^2}^{1/2}.
\]
Here $\widetilde{c}_0=\sqrt{2}c_0/\sqrt{\pi}$.

\smallskip
Then a consequence of Theorem \ref{theorem2.2} is
\begin{theorem}\label{theorem3.1}
Let $B$ a measurable set of $[0,T]$ of positive Lebesgue measure and $m>0$. There exists a constant $C>0$, that can depend only on $B$, $x_0$ and $m$, so that, for any $f\in mB_{H^2((0,\pi ))}$,
\[
 \|f\|_{L^2((0,\pi ))} \leq C\left( \left|\ln \|u_f (x_0,\cdot )\|_{L^\infty (B)}\right|^{-1}+\|u_f(x_0,\cdot )\|_{L^\infty (B)}\right).
 \]
\end{theorem}

We extend the previous result to a fractional one dimensional heat equation. To this end, for $\alpha >0$, we define $A^\alpha$, the fractional power of the operator $A=-\partial_x^2$ under Dirichlet boundary condition, as follows
\begin{align*}
&A^\alpha f=\frac{2}{\pi}\sum_{k\geq 1}k^{2\alpha}\widehat{f}_k\sin (kx),
\\
&D(A^\alpha )=\{f\in L^2((0,\pi ));\; \sum_{k\geq 1}k^{4\alpha}|\widehat{f}_k|^2<\infty\}.
\end{align*}

The IBVP for the heat equation for the fractional one dimensional Laplacian is represented by the Cauchy problem
\begin{equation}\label{3.6}
\left\{
\begin{array}{lll}
(\partial _t+A^\alpha )u=0\;\; \mbox{in}\; (0,+\infty ),
\\
u(\cdot ,0)=f.
\end{array}
\right.
\end{equation}
The solution of this Cauchy problem is given by, where $f\in L^2((0,\pi ))$,
\[
u_f^\alpha (x,t)=\frac{2}{\pi} \sum_{k\geq 1}e^{-k^{2\alpha}t}\widehat{f}_k\sin (kx).
\]

If $1/2<\alpha <1$, we can apply again Theorem \ref{theorem2.2}. We get

\begin{theorem}\label{theorem3.1bis}
We assume $1/2<\alpha <1$. Let $B$ a Lebesgue-measurable set of $[0,T]$ of positive Lebesgue measure and $m>0$. There exists a constant $C>0$, that can depend only on $B$, $x_0$, $\alpha$ and $m$,  so that, for any $f\in mB_{H^2((0,\pi ))}$,
\[
 \|f\|_{L^2((0,\pi ))} \leq C\left( \left|\ln \|u_f^\alpha (x_0,\cdot )\|_{L^\infty (B)}\right|^{-1}+\|u_f^\alpha (x_0,\cdot )\|_{L^\infty (B)}\right).
 \]
\end{theorem}

Next we consider the case $0<\alpha \leq1/2$. Since $\sum_{k\geq 1}k^{-2\alpha}=\infty$ , Theorem \ref{theorem3.1bis} is no longer valid in the present case. We are going to apply Theorem \ref{theorem2.4} instead of Theorem \ref{theorem2.2}. 

\smallskip
We pick $f\in H^{\theta +1}((0,\pi ))$ for some $\theta >1/2$. Let
\[
a_k=\frac{2}{\pi}\sin (kx_0)\widehat{f}_k,\;\; k\geq 1.
\]
In light of \eqref{3.4},  we get by using Cauchy-Schwarz's inequality 
\begin{equation}\label{3.7}
\sum_{k\geq 1} |\widehat{f}_k|\leq c_0\left( \sum_{k\geq 1}k^2|a_k|\right)^{1/2}\left( \sum_{k\geq 1}|a_k|\right)^{1/2}.
\end{equation}
But
\begin{equation}\label{3.7'}
\sum_{k\geq 1}k^2|a_k| \leq \left( \sum_{k\geq 1}\langle k\rangle ^{-2\theta } \right)^{1/2}\left( \frac{4}{\pi ^2}\sum_{k\geq 1}\langle k\rangle ^{2(\theta +1)}|\widehat{f}|^2 \right)^{1/2}.
\end{equation}
This and the first inequality in \eqref{3.4} imply
\begin{equation}\label{3.8}
\sum_{k\geq 1}k^2|a_k|\leq c_\theta \|f\|_{H^{\theta +1}((0,\pi ))}.
\end{equation}
Here and henceforth $c_\theta $ is a constant that can depend only on $\theta $.

\smallskip
Now a combination of \eqref{3.7} and \eqref{3.8} entails
\begin{equation}\label{3.9}
\sum_{k\geq 1} |\widehat{f}_k|\leq c_\theta \|f\|_{H^{\theta +1}((0,\pi ))}^{1/2}\|a\|_{\ell ^1}^{1/2}.
\end{equation}

As $H^{\theta +1}((0,\pi ))$ is continuously embedded in $C([0,\pi ])$,
\begin{equation}\label{3.10}
\sum_{k\geq 1} |\widehat{f}_k|^2\leq \sup_k|\widehat{f}_k|\sum_{k\geq 1} |\widehat{f}_k|\leq 2\|f\|_\infty \sum_{k\geq 1} |\widehat{f}_k|\leq c_\theta \|f\|_{H^{\theta +1}((0,\pi ))}\sum_{k\geq 1} |\widehat{f}_k|.
\end{equation}

Hence, it follows from \eqref{3.9} and \eqref{3.10} that
\begin{equation}\label{3.11}
\|f\|_{L^2((0,\pi ))}\leq c_\theta  \|f\|_{H^{\theta +1}((0,\pi ))}^{3/4}\|a\|_{\ell ^1}^{1/4}.
\end{equation}

Similarly to \eqref{3.7'}, we prove
\[
\|(\widehat{f}_k)\|_{\ell ^1}\leq c_\theta \|f\|_{H^\theta  ((0,\pi ))}.
\]
Thus
\begin{equation}\label{3.12}
\|a\|_{\ell ^1}\leq c_\theta \|f\|_{H^\theta  ((0,\pi ))}
\end{equation}
by \eqref{3.4}.

\smallskip
On the other hand, we have from \eqref{3.8}
\begin{equation}\label{3.13}
\|a\|_{\ell ^{1,2}}\leq c_\theta \|f\|_{H^{\theta +1} ((0,\pi ))}.
\end{equation}

In light of \eqref{3.11}, \eqref{3.12} and \eqref{3.13}, we obtain by applying Theorem \ref{theorem2.4}

\begin{theorem}\label{theorem3.2}
Let $m>0$ and $\theta >1/2$. There exists a constants $C>0$, that can depend only on $\theta$, $\alpha$, $x_0$ and $m$, so that, for any $f\in mB_{H^{\theta +1}((0,\pi ))}$,
\[
\|f\|_{L^2((0,\pi ))}\leq C\left( \left| \ln \left| \ln \left(m^{-1}c_\theta^{-1} \|u_f^\alpha (x_0,\cdot )\|_\infty \right)  \right|\right|^{-1/4}+\|u_f^\alpha (x_0,\cdot )\|_\infty \right).
\]
Here $c_\theta$ is the constant in \eqref{3.8}.
\end{theorem}

We observe that \eqref{2.12} is satisfied for the sequence $(n^{2\alpha})$ when $0<\alpha \leq 1/2$. The case $\alpha =1/2$ is obvious and for the case $0<\alpha <1/2$ it is a consequence of the following elementary inequality
\[
(n+1)^{2\alpha}-n^{2\alpha}=\frac{1}{2\alpha}\int_n^{n+1}\rho ^{2\alpha -1}d\rho \geq \frac{1}{2\alpha}\frac{1}{(n+1)^{1-2\alpha}}.
\]

We mention that is  possible to get a H\"older stability estimate even when $0<\alpha \leq 1/2$. To do that, we apply Theorem \ref{theorem2.5} instead of Theorem \ref{theorem2.4}. We leave to the interested reader to write down the details.


\subsection{Boundary measurement}

Let $\alpha >0$. We recall that the solution of the fractional heat equation \eqref{3.6} is given by, where $f\in L^2((0,\pi ))$,
\[
u_f^\alpha (x,t)=\frac{2}{\pi}\sum_{k\geq 1}e^{-k^{2\alpha}t}\widehat{f}_k\sin (kx).
\]
Since
\[
u_f^\alpha (x,T)=\frac{2}{\pi}\sum_{k\geq 1}e^{-k^{2\alpha}T}\widehat{f}_k\sin (kx),
\]
we get by applying Parseval's inequality
\begin{equation}\label{4.1}
\|u_f^\alpha (\cdot ,T)\|_{L^2((0,\pi ))}^2=\sum_{k\geq 1}e^{-2k^{2\alpha}T}|\widehat{f}_k|^2.
\end{equation}

On the other hand,
\begin{equation}\label{4.2}
\partial _xu_f^\alpha (0,t)=\frac{2}{\pi}\sum_{k\geq 1}k\widehat{f}_ke^{-k^{2\alpha}t}.
\end{equation}
When $\alpha >1/2$, $(\psi _k)$, the biorthognal set in $L^2(0,T)$ to $(e^{-k^{2\alpha}t})$, satisfies, for some constant $C>0$ depending on $\alpha$,
\[
\|\psi _n\|_{L^2((0,1))}\leq Ce^{Ck}.
\]
This inequality is obtained from \cite[estimate $(3.25)$]{FR}.

\smallskip
Therefore we have, similarly to \eqref{2.2.1},
\[
k^2|\widehat{f}_k|^2\leq Ce^{Ck}\|\partial _x u_f^\alpha (0,\cdot )\|_{L^2((0,T))}^2
\]
and then
\begin{align}
k^2e^{-2k^{2\alpha}T}|\widehat{f}_k|^2&\leq Ce^{Ck-2k^{2\alpha}T}\|\partial _x u_f^\alpha (0,\cdot )\|_{L^2((0,T))}^2\label{4.3}
\\
&\le C\|\partial _x u_f^\alpha (0,\cdot )\|_{L^2((0,T))}^2.\nonumber
\end{align}

Estimate \eqref{4.3} implies the following observability inequality
\begin{align}
\|u_f^\alpha (\cdot ,T)\|_{L^2((0,\pi ))}&=\left(\sum_{k\geq 1}e^{-2k^{2\alpha}T}|\widehat{f}_k|^2\right)^{1/2}\label{4.4}
\\
&\leq C\|\partial _x u_f^\alpha (0,\cdot )\|_{L^2((0,T))}.\nonumber
\end{align}

Let $B$ a Lebesgue-measurable set of $(0,T)$ of positive Lebesgue measure. As $\partial _x u_f^\alpha (0,\cdot )$ is given by a Dirichlet series, we get from Corollary \ref{corollary2.1} 
\begin{align}
\|u_f^\alpha (\cdot ,T)\|_{L^2((0,\pi ))}&=\left(\sum_{k\geq 1}e^{-2k^{2\alpha}T}|\widehat{f}_k|^2\right)^{1/2}\label{4.5}
\\
&\leq C\|\partial _x u_f^\alpha (0,\cdot )\|_{L^\infty (B)},\nonumber
\end{align}
under the condition that $(\widehat{f}_k)\in \ell ^1$.

\smallskip
Let $\widehat{g}_k$ be the $k$-th Fourier coefficient of $u_f^\alpha (\cdot ,T)$. Then
\[
\widehat{f}_k=e^{k^{2\alpha}T}\widehat{g}_k,\;\; k\geq 1.
\]
Hence, for any $N\geq 1$,
\[
\sum_{k\leq N}|\widehat{f}_k|^2\leq Ne^{N^{2\alpha}T}\sum_{k\leq N}|\widehat{g}_k|^2\leq Ne^{N^{2\alpha}T}\sum_{k\leq N}\|u_f^\alpha (\cdot ,T)\|^2_{L^2((0,\pi ))}.
\]
In light of \eqref{4.5}, this estimate yields
\[
\sum_{k\leq N}|\widehat{f}_k|^2\leq Ce^{CN^{2\alpha}}\|\partial _x u_f^\alpha (0,\cdot )\|_{L^\infty (B)}.
\]
Assuming in addition that $f\in mB_{H^\beta ((0,\pi))}$, for some $\beta >1/2$, we get
\[
\|f\|^2_{L^2((0,\pi))}\leq Ce^{CN^{2\alpha}}\|\partial _x u_f^\alpha (0,\cdot )\|_{L^\infty (B)}+\frac{m^2}{N^{2\beta}}.
\]
Here we used the fact that if $f\in H^\beta ((0,\pi))$, with $\beta >1/2$, then $(\widehat{f}_k)\in \ell ^1$.

\smallskip
As before, this estimate allows us to prove the following theorem.

\begin{theorem}\label{theorem4.1}
Let $B$ a Lebesgue-measurable set of $[0,T]$ of positive Lebesgue measure, $\beta >1/2$ and $m>0$. There exists a constant $C>0$, that can depend only on $B$, $\alpha$, $\beta$ and $m$,  so that, for any $f\in mB_{H^\beta((0,\pi ))}$,
\begin{equation}\label{4.6}
 \|f\|_{L^2((0,\pi ))} \leq C\left(\left|\ln \|\partial _x u_f^\alpha (0,\cdot )\|_{L^\infty (B)}\right|^{-\frac{\beta}{\max (\alpha ,\beta )}}+\|\partial _x u_f^\alpha (0,\cdot )\|_{L^\infty (B)}\right).
 \end{equation}
\end{theorem}

We observe that if instead of \eqref{4.5} we use \eqref{4.4}, then we get a variant of Theorem \ref{theorem4.1} in which \eqref{4.6} is substituted by 
\[
 \|f\|_{L^2((0,\pi ))} \leq C\left(\left|\ln \|\partial _x u_f^\alpha (0,\cdot )\|_{L^2((0,T))}\right|^{-\frac{\beta}{\max (\alpha ,\beta )}}+\|\partial _x u_f^\alpha (0,\cdot )\|_{L^2((0,T))}\right),
 \]
 without the restriction that $\beta >1/2$. We have only to assume that $\beta >0$.
 

\section{multidimensional case}

\subsection{An application of the one dimensional case}

We firstly recall that a finite or infinite sequence of real numbers is said to be non-resonant if every nontrivial rational linear combination of finitely many of its elements is different from zero. 

\smallskip
Let $\Omega = \prod_{i=1}^d(0,\mu _i\pi )$, where the sequence $(\mu _1,\ldots \mu _d)$ is non-resonant. From \cite[Proposition 5]{PS} the Dirichlet-Laplacian on $\Omega$ has simple eigenvalues
\[
\lambda_K=\prod_{i=1}^d\frac{k_i^2}{\mu _i^2},\;\; K=(k_1,\ldots ,k_d)\in \mathbb{N}^k,\; k_i\geq 1.
\]
To each $\lambda _K$ corresponds the eigenfunction 
\[
\varphi _K= \left(\frac{2}{\pi}\right)^{d}\frac{1}{\prod_{i=1}^d\mu _i}\prod_{i=1}^d\sin(k_ix_i/\mu_i)
\]
so that $(\varphi _K)$ forms an orthonormal basis of $L^2(\Omega )$.

\smallskip
Let $A:L^2(\Omega) \rightarrow L^2(\Omega )$  be the unbounded operator given by $A=-\Delta$ and $D(A)=H^2(\Omega )\cap H_0^1(\Omega)$. The fractional power $A^\alpha$, $\alpha >0$, is defined as follows
\begin{align*}
&A^\alpha f=\sum_{K=(k_1,\ldots k_d)\in \mathbb{N}^d,\; k_i\geq 1}\lambda _K^\alpha (f,\varphi_K)\varphi _K.
\\
&D(A^\alpha )=\left\{ f\in L^2(\Omega );\; \sum_{K=(k_1,\ldots k_d)\in \mathbb{N}^d,\; k_i\geq 1}\lambda _K^{2\alpha} |(f,\varphi_K)|^2<\infty \right\}.
\end{align*} 

We consider the Cauchy problem for the fractional heat equation associated to $A^\alpha$:
\begin{equation}\label{3.14}
\left\{
\begin{array}{lll}
(\partial _t+A^\alpha )u=0\;\; \mbox{in}\; (0,+\infty ),
\\
u(\cdot ,0)=f.
\end{array}
\right.
\end{equation}
 The solution of this Cauchy problem is given by
 \[
 u_f^\alpha (x,t)=\sum_{K=(k_1,\ldots k_d)\in \mathbb{N}^d,\; k_i\geq 1}e^{-t\lambda _K^\alpha} (f,\varphi_K)\varphi _K.
 \]
 
To reduce the multidimensional case to the one dimensional case, we need to restrict the initial sources to those of the form $f=f_1\otimes \ldots \otimes f_d\in \overset{d}{\underset{i=1}\otimes}C_0^\infty (0,\mu _i\pi )$. In that case,
 \[
 u_f^\alpha (x,t)=\prod_{i=1}^d \sum_{k_i\geq 1}e^{-tk_i^{2\alpha} /\mu_i^{2\alpha} } (f_i,\varphi_{k_i})\varphi _{k_i},
 \]
where 
\[
\varphi_{k_i}=\frac{2}{\mu _i\pi}\sin \left( \frac{k_ix_i}{\mu _i} \right).
\]
In other words,
\begin{equation}\label{3.15}
 u_f^\alpha (x,t)=\prod_{i=1}^du_{f_i}^\alpha (x_i,t).
\end{equation}
Here $u_{f_i}^\alpha$ is the solution of the one dimensional fractional heat equation \eqref{3.6} when $(0,\pi )$ is substituted by $(0,\mu_i \pi)$.

\smallskip
According to the parabolic maximum principle (see for instance \cite{RR}), we have 
\begin{equation}\label{3.16}
\|u_{f_i}^\alpha \|_{L^\infty ((0,\mu _i\pi)\times (0,T))}=\|f_j\|_{L^\infty ((0,\mu_i \pi ))}.
\end{equation}
Let us assume that 
\begin{equation}\label{3.17}
\inf_j \|f_j\|_{L^\infty ((0,\mu_i \pi ))}:=\eta >0.
\end{equation}
Henceforth, $x_0$ and $c_0$ are the same as in \eqref{3.3}. In light of \eqref{3.15}, \eqref{3.16} and \eqref{3.17}, we get
\[
\| u_f^\alpha (\cdot,\ldots ,\cdot ,\mu _jx_0,\cdot,\ldots ,\cdot )\|_{L^\infty \left(\prod_{i\neq j}(0,\mu_i \pi)\times (0,T)\right)}\geq \eta ^{d-1}\|u_{f_j}^\alpha(\mu_jx_0 ,\cdot )\|_{L^\infty ((0,T))}
\]
and then
\begin{align}
\Lambda (f):=\max_{1\leq j\leq d}\| u_f^\alpha (\cdot,\ldots ,\cdot ,\mu _jx_0,\cdot,\ldots ,\cdot )&\|_{L^\infty \left(\prod_{i\neq j}(0,\mu_i \pi)\times (0,T)\right)}\label{3.18}
\\
&\geq \eta ^{d-1}\|u_{f_j}^\alpha(\mu_jx_0 ,\cdot )\|_{L^\infty ((0,T))}.\nonumber
\end{align}

We fix $m>0$ and $\alpha \geq 1$. We prove similarly to Theorem \ref{theorem3.1} that there exists a constant $C>0$, that can depend only on $m$ and $\alpha$, so that, for any $f_i\in B_{H^2((0,\mu_i\pi))}$, $1\leq i\leq d$,
\[
 \|f_i\|_{L^2((0,\mu _i \pi))} \leq C \left|\ln \|u^\alpha_{f_i}(\mu_ix_0,\cdot )\|_{L^\infty ((0,T))}\right|^{-1}
\]
if $\Gamma (f)$ is sufficiently small. Hence, in light of \eqref{3.18},  there is $\Lambda _0>0$ such that
\begin{equation}\label{3.19}
 \|f_i\|_{L^2((0,\mu _i \pi))} \leq C \left|\ln \left( \eta ^{d-1}\Lambda (f)\right)\right|^{-1},\;\; \Lambda (f)\leq \Lambda_0 .
\end{equation}

From estimate \eqref{3.19} we get in a straighforward manner that

\begin{equation}\label{3.20}
 \|f\|_{L^2(\Omega )}=\prod_{i=1}^d\|f_i\|_{L^2((0,\mu _i \pi))} \le dC \left( \left|\ln \left( \eta ^{d-1}\Lambda (f)\right)\right|^{-1}+\Lambda (f)\right).
\end{equation}

A continuity argument enables us to extend estimate \eqref{3.20} the closure of  $\overset{d}{\underset{i=1}\otimes}C_0^\infty (0,\mu _i\pi )$ in $H^{2+(d-1)/2}(\Omega )$.

\smallskip
The case $\alpha <1$ can be treated similarly by using Theorems \ref{theorem2.4} and \ref{theorem2.5} instead of Theorem \ref{theorem2.2}.


\subsection{Boundary or internal measurement}

As we said in the introduction, there are only few results in the literature dealing with the problem of determining the initial heat distribution in a multidimensional heat equation from an overdermined data. The usual overspecified  data consists in an internal or a boundary measurement. We describe and comment briefly the main existing results and show the possible improvements.

\smallskip
Let $\Omega$ be a bounded domain of $\mathbb{R}^d$ with $C^2$-smooth boundary $\Gamma$. Let  $0<\lambda _1\leq \lambda _2\leq \ldots \lambda _k \leq \ldots$ be the sequence of eigenvalues, counted according to their multiplicity,  of the unbounded operator defined on $L^2(\Omega )$ by $A=-\Delta$ and $D(A)=H_0^1(\Omega )\cap H^2(\Omega )$. Let $(\phi _k)$ the corresponding sequence of eigenfunctions, chosen so that it forms an orthonormal basis of $L^2(\Omega )$.

\smallskip
By \cite[Theorem 1.43, page 27]{Ch}, for any $f\in H_0^1(\Omega )$, the IBVP for the heat equation 
\begin{equation}\label{1.4}
\left\{
\begin{array}{lll}
(\partial _t-\Delta )u=0\;\; \mbox{in}\; Q=\Omega \times (0,T),
\\
u=0\;\; \mbox{on}\; \Sigma=\Gamma \times (0,T),
\\
u(\cdot ,0)=f.
\end{array}
\right.
\end{equation}
has a unique solution $u_f\in H^{2,1}(Q)=L^2(0,T,H^2(\Omega))\cap H^1(0,T,L^2(\Omega ))$. Moreover, it follows from \cite[Theorem 1.42, page 26]{Ch} that $\partial _\nu u_f\in L^2(\Sigma )$.

\smallskip
Let $\gamma$ be a non empty open subset of $\Gamma$ and $\omega$ be a non empty open subset of $\Omega$. We set $\Sigma _\gamma =\gamma \times (0,T)$ and $Q_\omega =\omega \times (0,T)$.

\smallskip
When $f\in H_0^1(\Omega )$, we have the following two final observability inequalities 
\begin{equation}\label{1.5}
\|u_f(\cdot ,T)\|_{H_0^1(\Omega )}\leq C\| \partial _\nu u_f\|_{L^2(\Sigma_\gamma )}
\end{equation}
and 
\begin{equation}\label{1.6}
\|u_f(\cdot ,T)\|_{H_0^1(\Omega )}\leq C\| u_f\|_{L^2(Q_\omega )}.
\end{equation}
Here $u_f$ is the solution of the IBVP \eqref{1.4} and $C$ is a constant independent  on $f$.

\smallskip
Inequality \eqref{1.5} follows from \cite[Proposition 3.5, page 170]{Ch} and \eqref{1.6} is proved similarly (a variant of estimate \eqref{1.6} with less regularity assumption was given in  \cite[Corollary 6.3]{LL}).

\smallskip
Let, for $s>1/2$,
\[
H_0^s(\Omega )=\{w\in H^s(\Omega );\; u=0\; \mbox{on}\; \Gamma \;\mbox{(in the trace sense)}\}.
\]

Following \cite{Fu} 
\[
H_0^{2\theta}(\Omega )=\left\{ w\in L^2(\Omega );\; \sum_{k\geq 1}\lambda _k ^{2\theta}|(w,\phi _k)|^2<\infty \right\},\;\; \textrm{if}\;\; 1/4<\theta <3/4.
\]

\medskip
From the proof of \cite[Theorem 3.6, page 173]{Ch}, we deduce in a straightforward manner the following result

\begin{theorem}\label{theorem1.1}
Let $1/2\leq \theta <3/4$ and $m>0$. Then there exists a constant $C>0$, depending on $\Omega$, $\gamma$ (resp. $\omega$), $\theta$ and $m$, so that, for any $f\in mB_{H_0^{2\theta}(\Omega )}$,
\begin{equation}\label{1.7}
\|f\|_{L^2(\Omega )}\leq C\left(\left| \ln \| \partial _\nu u_f\|_{L^2(\Sigma_\gamma )}\right|^{-\theta}+\| \partial _\nu u_f\|_{L^2(\Sigma_\gamma )}\right)
\end{equation}
and
\begin{equation}\label{1.8}
\|f\|_{L^2(\Omega )}\leq C\left( \left| \ln  \| u_f\|_{L^2(Q_\omega )}\right|^{-\theta}+\| u_f\|_{L^2(Q_\omega )}\right).
\end{equation}
\end{theorem}

\begin{remark}\label{remark4.1}
1) According to \cite[Theorem 1]{AEWZ}, we can replace in \eqref{1.8}, $\| u_f\|_{L^2(Q_\omega )}$ by $\| u_f\|_{L^2(D)}$, where $D$ is any Lebsegue-measurable set contained in $\Omega\times (0,T)$, having a non zero Lebesgue measure. We can also improve the estimate \eqref{1.7} when $\partial \Omega$ contains a real-analytic open sub-manifold, that we denote by $\Gamma_a$. In light of \cite[Theorem 2]{AEWZ}, estimate \eqref{1.7} holds true if $\Sigma _\gamma$ is substituted by any Lebesgue-measurable subset of $\Gamma_a \times (0,T)$ with non zero Lebesgue measure. 

\smallskip
We note that the observability inequalities  appearing in \cite[Theorems 1 and 2]{AEWZ} hold for bounded domains $\Omega$ which are Lipschitz and locally star-shaped.

\smallskip
2) From \cite[Remark 6.1]{FZ},  there exist two constants $C_1>0$ and $C_2>0$, depending only on $\Omega$, $\omega$ and $T$, so that
\begin{equation}\label{1.9}
\sum_{n\geq 1}e^{-C_1\sqrt{\lambda _n}}|\widehat{f}_n|^2\leq C_2\int_{Q_\omega}|u_f|^2dxdt,\;\; f\in L^2(\Omega ).
\end{equation}
Here $\widehat{f}_n=\int_\Omega f\phi _ndx$.

On the other hand, since there is a constant $c>1$ such that $c^{-1}n^{2/d}\leq \lambda _n\leq cn^{2/d}$, $n\geq 1$, the inequality \eqref{1.9} is equivalent to the following one
\begin{equation}\label{1.10}
\sum_{n\geq 1}e^{-C_1n^{1/d}}|\widehat{f}_n|^2\leq C_2\int_{Q_\omega}|u_f|^2dxdt,\;\; f\in L^2(\Omega ).
\end{equation}

\smallskip
Clearly, the mapping 
\[
f\rightarrow \|f\|_{L_w^2(\Omega )}=\left(\sum_{n\geq 1}e^{-C_1n^{1/d}}|\widehat{f}_n|^2\right)^{1/2}
\]
defines a norm on $L^2(\Omega )$, weaker than the usual norm on $L^2(\Omega )$. Therefore, \eqref{1.10} can be reinterpreted as a Lipschitz stability estimate of determining $f$ from $u_f{_{|Q_\omega}}$:
\[
\|f\|_{L_w^2(\Omega )}\leq C_2\| u_f\|_{L^2(Q_\omega )},\;\; f\in L^2(\Omega ).
\]

A consequence of \eqref{1.10} is
\[
|\widehat{f}_n|^2\leq C_2e^{C_1n^{1/d}}\int_{Q_\omega}|u_f|^2dxdt,\;\; f\in L^2(\Omega ),\;\; n\geq 1.
\]
This estimate allows us to retrieve the estimate \eqref{1.8}.

\smallskip
3) Let us show that, in the case of an internal measurement, we can directly get a stability estimate without using the observability inequality \eqref{1.6}. The key of this direct proof relies on Lebeau-Robbiano type inequality for the eigenfunctions $\phi _n$. For $f\in L^2(\Omega )$, we set
\[
u_N(\cdot ,t)=\sum_{n=1}^N e^{-\lambda _nt}\widehat{f}_n\phi_n .
\]
Let $\omega$ be a Lebesgue-measurable subset of $\Omega$ of positive Lebesgue measure. From \cite[Theorem 5]{AEWZ}, we have
\[
\sum_{n=1}^N e^{-2\lambda _nt}|\widehat{f}_n|^2\leq Ce^{C\lambda _N}\int_{\omega}|u_N(\cdot ,t)|^2dx.
\]
But
\[
\int_{\omega}|u_N(\cdot ,t)|^2dx\leq \int_{\omega}|u_f(\cdot ,t)|^2dx+\sum_{n\geq N+1}|\widehat{f}_n|^2.
\]
Hence
\[
\sum_{n=1}^N e^{-2\lambda _nt}|\widehat{f}_n|^2\leq Ce^{C\lambda _N}\left[\int_{\omega}|u_f(\cdot ,t)|^2dx+\sum_{n\geq N+1}|\widehat{f}_n|^2\right].
\]
We integrate, with respect to $t$, between $0$ and $T$. We get in a straightforward manner that
\[
\sum_{n=1}^N |\widehat{f}_n|^2\leq Ce^{C\lambda _N}\left[\int_{Q_\omega}|u_f|^2dx+\sum_{n\geq N+1}|\widehat{f}_n|^2\right],
\]
implying that
\begin{align*}
\sum_{n\geq 1} |\widehat{f}_n|^2&\leq Ce^{C\lambda _N}\left[\int_{Q_\omega}|u_f|^2dx+\sum_{n\geq N+1}|\widehat{f}_n|^2\right]
\\
&\le Ce^{CN^{2/d}}\left[\int_{Q_\omega}|u_f|^2dx+\sum_{n\geq N+1}|\widehat{f}_n|^2\right].
\end{align*}
Therefore, under the assumption
\[
\sum_{n\geq 1} e^{cn ^\gamma }|\widehat{f}_n|^2\leq m,
\]
for some $c>0$, $m>0$ and $\gamma >d/2$, we obtain similarly to Theorem \ref{theorem2.5} the following H\"older stability estimate
\[
\|f\|_{L^2(\Omega )}\leq C\left( \|u_f\|_{Q_\omega}^\theta+ \|u_f\|_{Q_\omega}\right).
\]
\end{remark}

We end this subsection by mentioning that a Lipschitz stability estimate was established in \cite{SY} when the space of the initial heat distribution is given by a Banach space, that we denote by $B_\mu$ in the sequel, built on the Bergman-Selberg space $H_\mu$, where $\mu>1/2$ is a parameter.

\begin{theorem}\label{theorem1.2}
We fix $x_0\in \mathbb{R}^d\setminus \overline{\Omega}$, $\mu \in (1,5/4)$ and we set 
\[
\gamma_0=\{x\in \Gamma ;\; (x-x_0)\cdot \nu (x)>0\}\;\; \mbox{and}\;\; \Sigma _0=\gamma _0\times (0,T).
\]
There is a constant $C$, that can depend only on $\Omega$, $x_0$ and $\mu$, so that
\[
C^{-1}\|f\|_{L^2(\Omega )}\leq \|\partial _\nu u_f\|_{B_\mu (\Sigma _0)}\leq C \|f\|_{H^2(\Omega )},\;\; f\in H^2(\Omega )\cap H_0^1(\Omega ).
\]
\end{theorem}
A detailed proof of this theorem is given in \cite{Ch}.

\medskip
{\bf Acknowledgments.} I have had many discussions with Enrique Zuazua during the preparation of this paper. His  comments were useful for improving the major part of this work.



\begin{thebibliography}{}

\end{thebibliography}


\bigskip
\small

\begin{thebibliography}{CS2}
 \frenchspacing

\bibitem[AN]{AN} {\sc K. Ammari} and {\sc S. Nicaise}, {\em Stabilization of elastic systems by collocated feedback}, Springer, Heidelberg, 2015.
 
\bibitem[AS]{AS}{\sc M. Asaduzzaman} and {\sc S. Saitoh}, {\em Inversion formulas for the Reznitskaya transform and stability of Lipschitz determination of initial heat distribution}, Appl. Anal.  77 (3-4) (2001), 343-350.
 
\bibitem[AEWZ]{AEWZ}{\sc J. Apraiz, L. Escauriaza, G. Wang and C. Zhang}, {\em Observability inequalities and measurable sets},  J. Eur. Math. Soc. (JEMS) 16 (11) (2014), 2433-2475. 

\bibitem[BE]{BE}{\sc P. Borwein and T. Erd\'elyi}, {\em Generalizations of M\"untz's theorem via a Remez-type inequality for M\"untz spaces}, J. Amer. Math. Soc. 10 (1997) 327-349.

\bibitem[CZ]{CZ}{\sc C. Castro and   E. Zuazua},
{\em Unique continuation and control for the heat equation
from an oscillating lower dimensional manifold},
  SIAM J. Cont. Optim. 43 (4) (2005), 1400-1434.

\bibitem[Ch]{Ch}{\sc M. Choulli}, {\em Une introduction aux probl\`emes inverses elliptiques et paraboliques}, Math\'ematiques et Applications, 65, Springer, 2009.
 
\bibitem[DZ]{DZ}{\sc R. De Vore and E. Zuazua}, {\em Recovery of initial temperature from discrete sampling}, Math. Models and Methods in Appl. Sci. 24 (12) (2014), 2487-2501.

\bibitem[DFM]{DFM}{\sc L. D. Drager, R. L. Foote and C. F. Martin}, {\em Observing the heat equation on a torus along a dense geodesic}, Sys. Sci. Math. Sci. 4 (2) (1991), 186-192.

\bibitem[EHH]{EHH} {\sc A. El Badia, T. Ha Duong and A. Hamdi}, {\em Identification of a point source in a linear advection- dispersion-reaction equation: Application to a pollution source problem}, Inverse Problems  21 (2005), 1121-1136.

\bibitem[FR]{FR}{\sc H. O. Fattorini and D. L. Russell}, {\em Exact controllability theorems for linear parabolic equations in one space dimension}, Arch. Ration. Mech. Anal. 43 (1971) 272-292.

\bibitem[FZ]{FZ}{\sc E. Fern\`andez-Cara and E. Zuazua}, {\em The cost of approximate controllability for heat equations: the linear case}, Adv. Differential Equations 5 (4-6) (2000), 465-514.

\bibitem[Fu]{Fu} {\sc D. Fujiwara}, {\em Concrete characterization of the domains of fractional 
powers of some elliptic differential operators of the second order},
Proc. Japan Acad. 43 (1967) 82--86.

\bibitem[Ga]{Ga} {\sc W. Gautschi}, {\em On inverse of Vandermonde and confluent Vandermonde matrices}, Numer. Math. 4 (1962), 117-123.

\bibitem[GLM1]{GLM1}{\sc D. S. Gilliam, J. R. Lund and C. F. Martin}, {\em A Discrete sampling inversion scheme for the heat equation}, Numer. Math. 54 (1989), 493-506.

\bibitem[GLM2]{GLM2}{\sc D. S. Gilliam, J. R. Lund and C. F. Martin}, {\em Inverse parabolic problems and discrete orthogonality}, Numer. Math. 59 (1991), 361-383.

\bibitem[GMM]{GMM}{\sc D. S. Gilliam, B. A. Mair and C. F. Martin}, {\em Determination of initial states of parabolic systems from discrete data}, Inverse Problems 6 (1990), 737-747.

\bibitem[GM]{GM}{\sc D. S. Gilliam and C. F. Martin}, {\em Discrete observability and Dirichlet series}, Syst. Control Lett. 9 (1987), 345-348.

\bibitem[Ka]{Ka}{\sc O. Kavian}, {\em Introduction \`a la th\'eorie des points critiques et applications aux probl\`emes elliptiques}, Math\'ematiques et Applications, 13, Springer, 1993.


\bibitem[LOT]{LOT}{\sc Y. Li, S. Osher and R. Tsai}, {\em Heat source identification based on $\ell ^1$ constrained minimization}, Inverse Problems Imaging 8 (1) (2014), 199-221.

\bibitem[LTCW]{LTCW}{\sc G. Li, Y. Tan, J. Cheng and X. Wang}, {\em Determining magnitude of groundwater pollution sources by data compatibility analysis}, Inverse Problem in Science and Engineering 14 (2006), 287-300.

\bibitem[LL]{LL}{\sc J. Le Rousseau and J. Lebeau}, {\em On Carleman estimates for elliptic and parabolic operators. Applications to unique continuation and control of parabolic equations}, ESAIM Control Optim. Calc. Var. 18 (3) (2012) 712-747.

\bibitem[MRT]{MRT}{\sc S. Micu, I. Roventa and M. Tucsnak}, {\em Time optimal boundary controls for the heat equation}, J. Funct. Anal. 263 (2012) 25-49.


\bibitem[NSS]{NSS}{\sc G. Nakamura, S. Saitoh and A. Syarif}, {\em Representations of initial heat distributions by means of their heat distributions as functions of time}, Inverse Problems 15 (1999) 1255-1261.

\bibitem[PS]{PS}{\sc Y.~Privat, and M.~Sigalotti}, {\em The squares of the Laplacian-Dirichlet eigenfunctions are generically linearly independent}, ESAIM Control Optim. Calc. Var. 16  (3) (2010), 794-805. 

\bibitem[RR]{RR} M. Renardy and R. C. Rogers, An introduction to partial differential equations,
Springer-Verlag, New York, 1993.

\bibitem[SY]{SY}{\sc S. Saitoh and M. Yamamoto}, {\em Stability of Lipschitz type in determination of initial heat distribution}, J. Inequal. and Appl. 1 (1) (1997) 73-83.


\bibitem[TSS]{TSS}{\sc V. K. Tuan, S. Saitoh and S. Saigo}, {\em Size of support of initial heat distribution in the 1d heat equation}, Appl. Anal. 74 (3-4) (2000), 439-446.

\bibitem[Sc]{Sc}{\sc L. Schwartz}, {\em \'Etude des sommes d'exponentielles}, Hermann, Paris, 1959.

 
\bibitem[Wi]{Wi}{\sc D. V. Widder}, {\em An introduction to transform theory}, Academic Press, New York, 1971.


 \bibliographystyle{abbrv}

\end{thebibliography}
\end{document}